\documentclass[a4paper,11pt]{article}
\usepackage[latin1]{inputenc}
\usepackage{color}
\usepackage{ulem}
\usepackage{multicol}
\usepackage{amssymb, amsmath, amsthm}
\usepackage{latexsym}
\usepackage{graphicx}  %%%% Insertar figuras
\input xy
\xyoption{all}

\usepackage{verbatim}

\newtheorem{theorem}{Theorem}[section]

\newtheorem{lemma}[theorem]{Lemma}

\newtheorem{proposition}[theorem]{Proposition}

\newtheorem{definition}[theorem]{Definition}

\newtheorem{remark}[theorem]{Remark}

\newtheorem{notation}[theorem]{Notation}

\newcommand{\A}{{\mathcal A}}

\newcommand{\N}{{\mathbb N}}

\newcommand{\PP}{{\mathbb P}}
\newcommand{\E}{{\mathbb E}}
\newcommand{\e}{{\varepsilon}}

\title{Generalization of a theorem of Erd\H{o}s and R\'{e}nyi on Sidon Sequences}

\author{Javier Cilleruelo, S\'{a}ndor Z. Kiss, Imre Z. Ruzsa, Carlos Vinuesa}

\date{\today}

\begin{document}

\maketitle

\begin{abstract}
Erd\H os and R\'{e}nyi claimed and Vu proved that for all $h \ge 2$ and
for all $\e > 0$, there exists $g = g_h(\e)$ and a sequence
of integers $A$ such that the number of ordered representations of
any number as a sum of $h$ elements of $A$ is bounded by $g$, and
such that $|A\cap [1,x]| \gg x^{1/h-\e}$.

We give two new proofs of this result. The first one consists of an
explicit construction of such a sequence. The second one is
probabilistic and shows the existence of such a $g$ that satisfies $g_h(\e) \ll
\e^{-1}$, improving the bound $g_h(\e) \ll \e^{-h+1}$ obtained by
Vu.

Finally we use the ``alteration method'' to get a better bound for
$g_3(\e)$, obtaining a more precise
estimate for the growth of $B_3[g]$ sequences.
\end{abstract}

\section{Introduction}
Given an integer $h\ge 2$, we say that a sequence of integers $A$ is a
$B_h[g]$ sequence if every integer $n$ has at most $g$
representations as a sum of $h$ elements of $A$. We will write
$$r_{h,A}(n) = \left| \{ (a_1, a_2, \ldots, a_h) \ | \ n=a_1+\cdots +a_h,\quad a_1\le \cdots \le a_h,\quad a_i \in A \} \right|,$$
Thus, $A$ is a $B_h[g]$ sequence if $r_{h,A}(n) \le g$ for every positive integer $n$.

\smallskip

As usual, $A(x)=|A\cap [1,x]|$ counts the number of positive
elements of $A$ less than or equal to $x$. The counting method
easily gives $A(x)\ll x^{1/h}$ for any $B_h[g]$ sequence. It is
believed that a $B_h[g]$ sequence $A$ cannot satisfy $A(x)\gg
x^{1/h}$. However it is only known when $(h,g)=(\text{even},1)$.

\smallskip

In a seminal paper, Erd\H os and R\'{e}nyi \cite{ER} proved, using the
probabilistic method (see, for example, \cite{A} for an excellent
exposition of the method), that for any $\e > 0$ there exists $g=g(\e)$
and a $B_2[g]$ sequence such that $A(x)\gg x^{1/2-\e}$. In this
paper they claimed (but did not prove) that the same method gives the
analogous result for $h\ge 3$:

\smallskip

\begin{theorem} \label{h}
For any $\e>0$ and $h \ge 2$, there exists $g=g_h(\e)$ and a
$B_h[g]$ sequence, $A$, such that $A(x) \gg x^{1/h-\e}$.
\end{theorem}

\smallskip

Probably they did not notice that when $h \ge 3$ two distinct
representations of an integer as a sum of $h$ numbers can share
common elements. If $x_1+\cdots +x_h=y_1+\cdots +y_h$ and $x_1=y_1$,
then the events $(x_1,\dots , x_h\in A)$ and $(y_1,\dots ,y_h\in A)$
are not independent (this cannot happen when $h=2$ since one equal
addend implies that the other is also equal). This phenomena makes
the cases $h\ge 3$  much more difficult than the case $h=2$.

\smallskip

In \cite{Vu} Vu gave the first correct proof of Theorem \ref{h}. He
used ideas from a paper of Erd\H os and Tetali \cite{ET} to solve a
similar problem for a related question. The key point is the use of
the \textit{``Sunflower lemma''} to prove that if an integer has
enough representations then we can select $g+1$ representations which are
disjoint. Now the probabilistic method works easily because we deal
with independent events. If we follow the details of the proof we
can see that Vu obtains $g_h(\e)\ll \e^{-h+1}$.

\smallskip

The aim of this paper is to present new proofs of Theorem \ref{h}
and to obtain better relations between $g$ and $\e$.

\smallskip

The first one consists of an explicit construction of the sequence
claimed in Theorem \ref{h}. We do it in Section 2.

\smallskip

The second one is a probabilistic but distinct and simpler proof
than that presented by Vu. We  do not use the \textit{``Sunflower
lemma''}, but a simpler one, and we get a better upper
bound for $g_h(\e)$. More precisely, we prove the next theorem in
Section 3.

\begin{theorem}\label{e}
For any $\e>0$ and $h \ge 2$, there exists $g=g_h(\e)\ll \e^{-1}$
and a $B_h[g]$ sequence $A$, such that $A(x)\gg x^{1/h-\e}$.
\end{theorem}

Actually we can check in the proof of the theorem above that we can take
any $g_h(\e) \ge 2^{h-3}h(h-1)!^2\e^{-1}$. The improvement of this theorem
affects to the cases $h \ge 3$, where we have to deal with not
independent events. Vu's proof only gives $g_h(\e)\ll \e^{-h+1}$
which is worse than our bound when $h\ge 3$. For the case $h=2$
Erd\H os and R\'{e}nyi proved that any $g_2(\e) > \frac{1}{2 \e} - 1$
satisfies the condition of Theorem \ref{e} and the first author
\cite{C} used the ``alteration method'' to improve that bound to
$g_2(\e) > \frac{1}{4 \e} - \frac{1}{2}$.

\bigskip

In the last section we refine Theorem \ref{e} when $h=3$ proving that
$g_3(\e) > \frac{2}{9\e} - \frac{2}{3}$ works. In other words,

\begin{theorem} \label{boundge}
For every $\e > 0$ and for every $g \ge 1$ there is a $B_3[g]$
sequence $A$, such that
$$A(x) \gg x^{\frac{g}{3g+2} - \e}.$$
\end{theorem}

It is also possible to refine Theorem \ref{e} for $h\ge 4$ using the
``alteration method'' but the exponents we would obtain in these cases
are not satisfactory enough. For ``satisfactory enough'' we mean exponents
such that when we particularize to $g=1$, we obtain the same exponent that
we get with the greedy algorithm. That is what happens for $h=2$ (\cite{C})
and $h = 3$ (Theorem \ref{boundge}).

\section{A constructive proof of Theorem \ref{h}}
Given $h$ and $\e$, we construct a sequence $A$ with $r_{h,A}(n)$ bounded and we
will prove that $A(n) > n^{1/h-{\varepsilon}}$ for sufficiently
large $n$, which implies Theorem \ref{h}.

\bigskip

We use the representation of natural numbers in a number system with
variable base. It is easy to see that every natural number $x$ can
be expressed uniquely in the form
\begin{equation*}
x = b_{0} + b_{1}q_{1} + b_{2}q_{1}q_{2} + \cdots + b_{s}q_{1}
\cdots q_{s} + \cdots,
\end{equation*}
\noindent where $0 \le b_{i} < q_{i+1}$. The $b_{i}$'s and $q_{i}$'s
are natural numbers, $b_{i}$'s called the ``digits'' and $q_{i}$'s
called the ``bases''.

\bigskip

We consider $l\ge 2$, a large enough number that will be fixed
later. We fix the sets $0 \in A_{i} \subset \left[0,
\frac{q_{i}}{h}\right)$ such that the $A_{i}$'s are maximal sets
with the condition $r_{h,A_i}(n) \le 1$ for every $n$. It is
known (see for example \cite{HR}) that if $p$ is a prime, there is a
$B_h[1]$ set in $[0,p^h-2]$ with $p$ elements. Combining this result
with Bertrand's postulate, we can assure that
\begin{equation}
\label{sizeofAi}|A_i|>\frac 12\left (\frac{q_i}h\right )^{1/h}.
\end{equation}

Now we construct the set $A$ in
the following way: put that natural numbers in $A$ which digits
$b_{i} \in A_{i+1}$, and for which there is an $m$ such that $b_{i}
= 0$ for $i \not \in [m+1, \dots{}, m+l]$.

\bigskip

First we prove that $r_{h,A}(n) < (h!)^{lh}$. We add up $h$ numbers, $a_1$, $a_2$, \ldots, $a_h$.

Since the $j$-th digit of each addend is in $\left [0, \frac{q_j}{h}
\right )$, each digit of the sum will be the sum of the $j$-th
digits of the $a_i$'s (in other words, there will be no carries).

And, since the $j$-th digit of each addend is in a $B_h[1]$ set, the $j$-th digit of the sum can be obtained in only one way as a sum of $h$ digits.
Note that $h$ numbers have $h!$ permutations, so for each digit of the sum we could have the corresponding digits of the $h$ addends distributed in at
most $h!$ ways.

Finally, observe that the sum of the number of non zero digits of all the addends is less than or equal to $hl$, so the number of digits of the sum
different from zero will also be less than or equal to $hl$, and finally we will have $r_{h, A} (n) \le (h!)^{lh}$ for every $n$.

\bigskip

Now, we give an estimation of the value of $A(n)$. Given $n$, we know that there
exists $j$ such that
\begin{equation} \label{q1qj}
q_{1}q_{2} \cdots q_{j} \le n < q_{1}q_{2} \cdots q_{j+1}.
\end{equation}

It is clear that those integers which digits
\[
b_{0} = b_{1} = \cdots = b_{j-l-1} = 0 \ \text{ and } \  b_{i} \in
A_{i+1}, \ i = j - l, \dots{}, j - 1,
\]
\noindent are in $A$. Let $N$ denote the number of such integers. We
define $r = \dfrac{\log_2 l}{l}$ and, for $i \ge 1$,
\begin{equation} \label{qi}
q_{i} = \lfloor e^{(1+r)^{i-1}} \rfloor.
\end{equation}

%We know (Bose-Chowla\footnote{
%If $q_i/h \le 1$ then (\ref{sizeofAi}) is true with $A_i=\{0\}$. If $q_i/h > 1$ then we have two subcases:
%\begin{itemize}
%\item $\frac{1}{2} (q_i/h + 1)^{1/h} \ge 2$: The result of Bose and Chowla implies
%that for every $h \ge 2$ and for every prime $p$ there is a set
%$C=\{c_1, \ldots, c_p\} \subseteq [0, p^h -1)$ with $r_{h,C}(n) \le 1$.
%The well known Bertrand's postulate says that for every integer $n \ge 2$
%there is a prime $p \in (n,2n)$. So, given $q_i$ and $h$, we take
%$z = \lfloor \frac{1}{2} (q_i/h + 1)^{1/h} \rfloor$ and we know by
%Bertrand's postulate (observe that $z \ge 2$) that there is a prime $p$
%such that $z < p < 2z$, i. e. $\frac{1}{2} (q_i/h + 1)^{1/h} \le p \le (q_i/h + 1)^{1/h}$.
%Finally, the result of Bose and Chowla says that there is a set $A_i$
%with $|A_i| > \frac{1}{2} (q_i/h)^{1/h}$ elements in $[0, p^h - 1) \subseteq [0, q_i/h)$ with
%$r_{h,A_i}(n) \le 1$ and (\ref{sizeofAi}) is true.
%\item $\frac{1}{2} (q_i/h + 1)^{1/h} < 2$: In this subcase, we have $\frac{1}{2} (q_i/h)^{1/h} < 2$
%and (\ref{sizeofAi}) is true with $A_i = \{0, 1\}$ (observe that $q_i/h > 1$).
%\end{itemize}
%}; see, for example, \cite{HR}) that
%\begin{equation} \label{sizeofAi}
%|A_{i}| > \frac 12\left (\frac{q_{i}}{h}\right )^{1/h}.
%\end{equation}
Since $\frac {e^{(1+r)^{i-1}}}2 \le q_{i}\le e^{(1+r)^{i-1}}$ and
$2(2h)^{1/h}\le 2e^{2/e}<e^2$, inequality (\ref{sizeofAi})
implies
\begin{equation}\label{Ai}
|A_{i}|>\frac 12\left (\frac{e^{(1+r)^{i-1}}}{2h}\right
)^{1/h}>e^{\frac{(1+r)^{i-1}}h-2}.
\end{equation}

\bigskip

First we give an upper bound for $\log n$. It follows from
(\ref{q1qj}) and (\ref{qi}) that
\begin{equation} \label{logqn}
\log n < \log (q_{1} \cdots q_{j+1}) \le 1 + (1+r) + \cdots +
(1+r)^{j} < \frac{(1+r)^{j+1}}{r}.
\end{equation}

\bigskip

In the next step we will give a lower estimation for $\log N$.
Applying (\ref{Ai}) we have
\begin{eqnarray} \label{logqN}
\log N=\sum_{i=j-l+1}^{j}\log |A_i| & > & \frac{(1+r)^{j-l} +
\cdots + (1+r)^{j-1}}{h}-2l \nonumber \\
& = & \frac{(1+r)^{j}}{hr}(1 - (1+r)^{-l}) - 2l.
\end{eqnarray}

In view of (\ref{logqn}) and (\ref{logqN}) we have
\begin{eqnarray} \label{hlN/ln}
\frac{h\log N}{\log n} & > & \frac{(1+r)^{j}(1 - (1+r)^{-l})
}{(1+r)^{j+1}}- \frac{2lrh}{(1+r)^{j+1}} \nonumber \\
& = & \frac{1 - (1+r)^{-l}}{1 + r} - \frac{2lrh}{(1+r)^{j+1}}.
\end{eqnarray}

Using that $\frac 1{1+r}>1-r$ and that $\left (1+\frac{\log_2
l}l\right )^{\frac l{\log_2 l}}\ge 2$ for any $l\ge 2$, we have
\begin{eqnarray}\label{des}
\frac{1 - (1+r)^{-l}}{1 + r}&>&1-r-(1+r)^{-l} =1-\frac{\log_2
l}l- \left( 1+\frac{\log_2 l}l \right)^{-l}\nonumber\\ &>&1-\frac{\log_2 l}l-\frac
1l>1-\frac {2\log_2 l}l.
\end{eqnarray}

On the other hand, since
$
\lim_{j\to\infty}\frac{2lrh}{(1+r)^{j+1}} = 0
$
we have that, for sufficiently large $j$,
\begin{equation}\label{j}
\frac{2lrh}{(1+r)^{j+1}} < \frac {\log_2 l}l.
\end{equation}

Finally, from (\ref{hlN/ln}), (\ref{des}) and (\ref{j}) we have
\[
\frac{h\log N}{\log n} > 1 - \frac {3\log_2 l}l,
\]
\noindent for sufficiently large $n$.

We finish the proof of Theorem \ref{h} taking, for a given $\e>0$,
a large enough integer $l$ such that $\frac {3\log_2 l}l<h\e$, because
then $\log N > \left( \frac{1}{h} - \e \right) \log n$, i. e. $N > n^{1/h - \e}$.

\smallskip

Just a little comment about the dependence of $g$ on $\e$. Observe
that our $g$ is $(h!)^{lh}$ and that, given $\e$, we need to choose
a large value of $l$, say $l\gg \e^{-1}\log \e^{-1}$. This makes the
dependence of $g$ on $\e$ very bad. The value of $g$ we get with
this construction depends more than exponentially on $\e^{-1}$. We
will try to improve this in the next section and, for the case $h =
3$, even more in the last one.

\smallskip

Note that in \cite{Ci2} we can find an explicit Sidon sequence with
$A(x) \gg x^{1/3-o(1)}$. But this construction can not be
generalized to get dense $B_h[g]$ sequences.

\section{A new probabilistic proof of Theorem \ref{h}}

\begin{definition}Given $0<\alpha<1$ we define $S(\alpha,m)$ as the probability
space of the sequences of positive integers defined by $$P(x\in
A)=\begin{cases} 0 &\text{ if } x<m
\\ x^{-\alpha} &\text{ if } x\ge m\end{cases}.$$
\end{definition}

\begin{theorem} \label{densi}
For any $m$, a random sequence $A$ in $S(\alpha,m)$ satisfies
$A(x)\gg x^{1-\alpha}$ with probability $1$.
\end{theorem}
\begin{proof}
First of all, we calculate
$$\E(A(x)) = \sum_{n \le x} \PP(n \in A) = \sum_{m \le n \le x} n^{- \alpha} = \frac{x^{1 - \alpha}}{1 - \alpha} + O_{\alpha, m}(1),$$
when $x \to \infty$. Now, we use Chernoff's Lemma\footnote{Let
$X=t_1+\cdots +t_n$ where the $t_i$ are independent Boolean random
variables. \textbf{Chernoff's Lemma} says that for every $0 < \e <
2$, $\PP(|X- \E(X)| \ge \e \E(X)) \le 2e^{-\e^2  \E(X) / 4}$.} to
get
$$\PP\left (A(x) \le \tfrac{1}{2} \E\left (A(x)\right )\right ) \le \PP \left (|A(x) - \E(A(x))| \ge \tfrac{1}{2} \E(A(x))\right ) \le
2 e^{- \frac{1}{16}\left (\frac{x^{1 - \alpha}}{1 - \alpha} +
O(1)\right )}.$$ Since $\sum_{x = 1}^{\infty} 2 e^{-
\frac{1}{16}\left (\frac{x^{1 - \alpha}}{1 - \alpha} + O(1)\right )}
< \infty$, Borel-Cantelli Lemma\footnote{\textbf{Borel-Cantelli
Lemma.} Let $(E_n)_{n=1}^{\infty}$ be a sequence of events in a
probability space. If $\sum_{n=1}^{\infty} \PP(E_n) < \infty$ then
the probability that infinitely many of them occur is $0$.}
 says that $A(x) \gg  x^{1-\alpha}$ with probability 1.
\end{proof}

\begin{notation}
We denote the set which elements are the coordinates of the vector
$\bar{x}$ as $Set(\bar{x})$. Of course, if two or more coordinates of $\bar{x}$ are equal, this value appears only once in
$Set(\bar{x})$.
\end{notation}

\begin{notation}
We define
$$R_h(n) = \{ (n_1, n_2, \ldots, n_h) \ | \ n=n_1+\cdots +n_h,\quad n_1\le \cdots \le n_h,\quad n_i \in \N \}.$$
\end{notation}

\begin{lemma} \label{Erla} For a sequence $A$ in $S(\alpha,m)$, for every $h$ and $n$
$$\E(r_{h,A}(n)) \le C_{h,\alpha} \  n^{h(1-\alpha) - 1}$$
where $C_{h, \alpha}$ depends only on $h$ and $\alpha$.
\end{lemma}

\begin{proof}
\begin{eqnarray*}
\E(r_{h, A}(n)) & = & \sum_{\bar{x} \in R_h(n)} \prod_{x \in
Set(\bar{x})} \PP(x \in A) \\
& = & \sum_{j=1}^{h} \sum_{\substack{\bar{x} \in R_h(n) \\
|Set(\bar{x})| = j}} \prod_{x \in Set(\bar{x})} \PP(x \in A).
\end{eqnarray*}
Since the largest element of every $\bar{x} \in R_h(n)$ is $\ge n/h$
we have
\begin{eqnarray*}
\E(r_{h, A}(n)) & \le & \left( \frac{n}{h}\right)^{- \alpha}
\sum_{j=1}^{h} \left( \sum_{x < n} x^{- \alpha} \right)^{j-1} \\
& \le & \left( \frac{n}{h}\right)^{- \alpha} \sum_{j=1}^{h} \left(
\int_{0}^{n} x^{- \alpha} \ dx \right)^{j-1}\\
 & \le & \left( \frac{n}{h}\right)^{- \alpha}
\sum_{j=1}^{h} \left( \frac{n^{1 - \alpha}}{1 - \alpha} \right)^{j-1} \\
& \le & C_{h, \alpha} \  n^{h(1-\alpha) - 1}.
\end{eqnarray*}
\end{proof}

\begin{definition}
We say that two vectors $\bar{x}$ and $\bar{y}$ are disjoint if
$Set(\bar{x})$ and $Set(\bar{y})$ are disjoint sets. We define
$r^*_{l, A} (n)$ as the maximum number of pairwise disjoint
representations of $n$ as sum of $l$ elements of $A$, i. e. the
maximum number of pairwise disjoint vectors of $R_{l}(n)$ with
their coordinates in $A$. We say that $A$ is a $B^*_l[g]$ sequence
if $r^*_{l,A}(n)\le g$ for every $n$.
\end{definition}

\begin{lemma} \label{Pfla} For a sequence $A$ in $S(\alpha,m)$, for every $h$ and $n$
$$\PP(r^*_{h, A}(n) \ge s) \le C_{h,\alpha, s} \  n^{(h(1-\alpha) - 1)s}$$
where $C_{h, \alpha, s}$ depends only on $h$, $\alpha$ and $s$.
\end{lemma}

\begin{proof}
Using the independence given by the pairwise disjoint condition
\begin{eqnarray*}
\PP(r^*_{h, A}(n) \ge s) & = & \sum_{\substack{\{\bar{x}_1, \ldots,
\bar{x}_s \} \\ \bar{x}_i \in R_h(n) \\ \bar{x}_1, \ldots,
\bar{x}_s \\ \text{ pairwise disjoint}}} \prod_{i=1}^{ s } \prod_{x \in
Set(\bar{x}_i)} \PP(x \in
\A) \\
& \le & \left( \sum_{\bar{x} \in R_h(n)} \prod_{x \in Set(\bar{x})}
\PP(x \in \A) \right)^{ s }\\
& = & E(r_{h, A}(n))^{ s }
\end{eqnarray*}
and using Lemma \ref{Erla} we conclude the proof.
\end{proof}

\bigskip
\begin{proposition} \label{fla}
Given $h \ge 2$ and $0<\e<1/h$, a random sequence in $S(1 - \frac{1}{h} + \e, m)$
is a $B^*_h [g]$ sequence for every $g \ge \frac{2}{h \e}$ with probability $1-O(\frac{1}{m})$.
\end{proposition}
\begin{proof}
>From Lemma \ref{Pfla}, and taking into account the value of
$\alpha= 1 -\frac{1}{h} +\e$, we have
$$\PP(r^*_{h, A}(n) \ge g+1) \le C_{h, \e, g} n^{- h \e (g+1)}.$$
Since $r^*_{h, A}(n)=0$ for $n<m$ we have
\begin{eqnarray*}
\PP \left (r^*_{h,A}(n)\ge  g+1 \text{ for some } n \right ) & \le &
\sum_{n\ge m} \PP \left (r^*_{h, A}(n) \ge g+1 \right )\\
& \le & \sum_{n\ge m}C_{h, \e}n^{-h\e (g+1)}.
\end{eqnarray*}
If $g\ge \frac 2{h\epsilon}$, the last sum is $O(1/m)$. Thus, if it
is the case,
\begin{eqnarray*}
\PP \left (r^*_{h,A}(n) \le g \text{ for every } n \right ) & = & 1 - \PP \left( r^*_{h,A}(n) \ge  g+1 \text{ for some } n \right) \\
& \ge & 1 - O \left ( \frac{1}{m} \right ).
\end{eqnarray*}
\end{proof}

The next ``simple'' lemma is the key idea of the proof of
Theorem \ref{e}.

\begin{lemma} \label{cap}
$$B^*_h[g] \cap B_{h-1}[k] \subseteq B_h[hkg].$$
\end{lemma}
\begin{remark} \label{remarkcap}
In fact, the ``true'' lemma, which is a little more ugly, is
$$B^*_h[g] \cap B_{h -1}[k] \subseteq B_h[g(h(k-1) + 1)]$$
and this is what we will prove. As an example, we can have in mind that a sequence $A$ with $r^*_{3,A}(n) \le g$ which is a Sidon sequence is also a $B_3[g]$ sequence, i. e. $B^*_3[g] \cap B_2[1] \subseteq B_3[g]$, since in this case two representations that share one element share the three of them.
\end{remark}

\begin{proof}
We proceed by contradiction. Suppose that $A \in B^*_h[g] \cap B_{h-1}[k]$ and suppose that there is an $n$ with $g(h(k-1) + 1)+1$ distinct representations as a sum of $h$ elements of $A$.

Fix one of these representations, say $n = x_1 + x_2 + \cdots + x_h$. How many representations of $n$ can intersect with it? Well, the number of representations of $n$ that involve $x_1$ is at most $k$, since $A \in B_{h-1}[k]$. So we have at most $k-1$ more representations with $x_1$. The same thing happens for $x_2, \ldots, x_h$. So, finally, the maximun number of representations that can intersect with the one we fixed is $h(k-1)$.

Now we fix a second representation of $n$ that does not intersect with the first one that we chose. Again, there are at most $h(k-1)$ representations that intersect with our new choice.

After $g$ disjoint choices, counting them and all the representations that intersect with each one of them, we have at most $gh(k-1) + g$ representations. By hipothesis, there is at least one representation of $n$ left that does not intersect with any of our $g$ choices. But this means that we have $g+1$ disjoint representations of $n$, which contradicts the fact that $A \in B^*_h[g]$.
\end{proof}

\begin{proposition} \label{c_h}
For every $h \ge 2$ and $0<\e<1/h$ a random sequence in $S(1 -
\frac{1}{h} + \e, m)$ is a $B_h[g]$ sequence for every $g \ge
c_h/\e$ with probability $1 - O(\frac{1}{m})$, where $c_h = 2^{h-3}h
(h-1)!^2$.
\end{proposition}
\begin{proof}
We proceed by induction on $h$.

For $h=2$, and using Proposition \ref{fla}, the result is true since a $B^*_2[g]$ sequence is the same that a $B_2[g]$ sequence.

Now suppose that the result is true for $h - 1$. Let $\alpha = 1 -
\frac{1}{h} + \e$. From Proposition \ref{fla} we know that a random
sequence in $S(\alpha , m)$ is $B_h^*[g_1]$ for every $g_1 \ge
\frac{2}{h\e}$ with probability $1 - O(\frac{1}{m})$. But, since
$\alpha > 1 - \frac {1}{h-1} + \frac{1}{h(h-1)}$, by the induction
hypothesis we know that this random sequence is also $B_{h-1}[g_2]$
for every $g_2 \ge h(h-1)c_{h-1} = c_h/2$ with probability $1 -
O(\frac{1}{m})$. So, with probability $1 - O(\frac{1}{m})$ the two
things happen at the same time, i. e. the random sequence is in
$B_h^*[g_1] \cap B_{h-1}[g_2]$ for every $g_1 \ge \frac{2}{h\e}$ and
$g_2 \ge c_h/2$.

Lemma \ref{cap} concludes the proof.
\end{proof}

Lemma \ref{densi} and Proposition \ref{c_h} imply Theorem \ref{e}.

\section{Sequences with $r_{3,A}(n)$ bounded}

%One way for obtaining a Sidon sequence (i. e. a sequence $A$ with
%$r_{2,A}(n) \le 1$ for every $n \in \N$) is using the well known
%``greedy algorithm''. Briefly, we construct the sequence $A =
%\{a_1, a_2, a_3\ldots\}$ inductively: we take $a_1 = 1$ and, once we
%have chosen $a_1, a_2, \ldots, a_{m-1}$ forming a Sidon set, $a_m$
%is the least positive integer different from $a_r + a_s - a_t$ with
%$1 \le r, s, t \le m - 1$. This sequence $A$ has $A(x) \gg
%x^{1/3}$. We can make the obvious generalization to obtain a
%sequence, $A$, with $r_{h,A}(n) \le 1$ and $A(x) \gg
%x^{1/{(2h-1)}}$.

Now\footnote{Of course, our Proposition \ref{c_h} gives a relation
between $g$ and $\e$, but observe that for $g=1$ it gives values of
$\e \ge 1$, so it does not give any useful information. In our
terminology, this is not a ``satisfactory enough exponent''.}, we will
try to find a more precise relation between $g$ and $\e$. In fact, the
result of Erd\H{o}s and R\'enyi in \cite{ER} is more precise than
what we said in the Introduction. They proved that for every $g
> \frac{1}{2 \e} - 1$ there is a  $B_2[g]$ sequence, $A$, with $A(x)
\gg x^{1/2 - \e}$. Stated perhaps in a more convenient way, what
they proved is that for every positive integer $g$ there is a
sequence $A$ such that $r_{2, A}(n) \le g$ with $A(x) \ge
x^{\frac{1}{2 + 2/g } - o(1)}$, as $x \rightarrow \infty$.
%For the special
%case $g=1$ this says that there is a Sidon sequence with $A(x) \gg
%x^{\delta}$ for every $\delta < 1/4$. This exponent is worse than
%the $1/3$ given by the ``greedy algorithm''.

In \cite{C} the first author used the ``alteration method'' (perhaps
our random sequences do not satisfy what we want but they do if we
remove ``a few'' elements) to prove that for every $g > \frac{1}{4
\e} - \frac{1}{2}$ there is a  $B_2[g]$ sequence, $A$, with $A(x)
\gg x^{1/2 - \e}$. In other words, for every positive integer
$g$ there is a sequence $A$ such that $r_{2, A}(n) \le g$ with
$A(x) \gg x^{\frac{1}{2 + 1/g }-o(1)}$ as $x \rightarrow \infty$.
%Observe that for $g = 1$ we have a Sidon sequence with $A(x) \gg
%x^{\delta}$ for every $\delta < 1/3$.

\smallskip

In this section we will use the ideas from \cite{C} to prove Theorem
\ref{boundge}, which is a refinement on the dependence between $g$ and
$\e$, for sequences with $r_{3,A}(n) \le g$.

\bigskip

\begin{definition} \label{bad}
Given a sequence of positive integers, $A$, we say that $x$ is
$(g+1)_h-$bad (for $A$) if $x \in A$ and there exist $x_1, \dots ,
x_{h-1}\in A$, $x_1 \le \cdots \le x_{h-1} \le x$, such that $r_{h, A}
(x_1 + \cdots + x_{h-1} + x) \ge g+1$.
\end{definition}

In other words, $x \in A$ is $(g+1)_h-$bad if it is the largest
element in a representation of an element that has more than $g$
representations as a sum of $h$ elements of $A$. Observe that $A$ is
a $B_h[g]$ sequence if and only if it does not contain $(g+1)_h-$bad
elements.

\begin{definition} \label{Btilde}
A sequence of positive integers, $A$, is in $\tilde{B}_h[g]$ if the
number of $(g+1)_h-$bad elements less than or equal to $x$ for $A$, say
$\mathcal B(x)$, is
$$\mathcal B(x) = o(A(x)) \text{ when } x \to \infty.$$
\end{definition}

So, $A \in \tilde{B}_h[g]$ if removing a few elements from it (``a
little o''), it is a $B_h[g]$ sequence.

\begin{notation}
We denote by $\mathcal B_{k,h}(g+1)$ the set of $(g+1)_h-$bad
elements for $A$ in the interval $[h^k, h^{k+1})$.
\end{notation}

Substituting $r_{h, A}$ by $r^*_{h, A}$, we define the $(g+1)^*_h-$bad elements for $A$.
Analogously, we define $A \in \tilde{B}^*_h[g]$ and $\mathcal B^*_{k,h}(g+1)$.

\bigskip

Obviously, the ``tilde'' version of Lemma \ref{cap} is also true. In particular, from Remark \ref{remarkcap}:
\begin{lemma} \label{captilde}
$\tilde{B}^*_3[g] \cap \tilde{B}_{2}[1] \subseteq \tilde{B}_3[g]$.
\end{lemma}

Now we can write the next theorem, which we will use only in the cases $h=2$ and $h=3$.

\begin{theorem} \label{Btilde*hg}
Given $0 < \delta < \frac{1}{2h - 3}$ and $h \ge 2$, a random sequence $A$ in $S \left( \frac{2h - 4}{2h - 3} + \delta, m \right)$ is $\tilde{B}^*_h[g]$ for every $g > \frac{\frac{h-1}{2h-3} - (h-1)\delta}{\frac{h - 3}{2h - 3} + h \delta}$ with probability $1 - O \left( \frac{1}{\log m} \right)$.
\end{theorem}

\begin{proof}
We consider a random sequence $A$ in $S(\alpha, m)$.
\begin{eqnarray*}
\E(|\mathcal B^*_{k,h}(g+1)|) & = & \sum_{h^k \le x < h^{k+1}} \PP(x \text{ is } (g+1)^*_h-\text{bad}) \\
& \le & \sum_{h^k \le x < h^{k+1}} \sum_{\substack{\bar{x}_{g+1}=(y_1, \ldots, y_{h-1},x) \\ y_1 \le \cdots \le y_{h-1} \le x}} \prod_{z \in Set(\bar{x}_{g+1})} \PP(z \in A) \cdot\\
&&\cdot \sum_{\substack{ \{ \bar{x}_1, \ldots, \bar{x}_g\} \\ \bar{x}_i \in R_h(y_1 + \cdots + y_{h-1} + x) \\ \bar{x}_1, \ldots, \bar{x}_g, \bar{x}_{g+1} \\ \text{pairwise disjoint}}} \prod_{i=1}^{g} \prod_{z \in Set(\bar{x}_i)} \PP(z \in A) \\
& \le & \sum_{h^k \le n < h^{k+2}} \left( \sum_{\bar{x} \in R_h(n)} \prod_{x \in Set(\bar{x})} \PP(x \in A) \right)^{g+1} \\
& \le & \sum_{h^k \le n < h^{k+2}} \left( C n^{h(1 - \alpha) - 1} \right)^{g+1} \\
& \ll & \sum_{h^k \le n < h^{k+2}} n^{(h - 1 - h \alpha) (g+1)} \\
& \ll & h^{k((h - 1 - h \alpha)(g+1) + 1)}
\end{eqnarray*}
when $k \to \infty$, where we have used Lemma \ref{Erla}.

Now, we can use Markov's Inequality\footnote{\textbf{Markov's Inequality.} For a random variable $X$ and $a > 0$, $\PP(|X| \ge a) \le \dfrac{\E(|X|)}{a}$.} to have:
$$\PP(|\mathcal B^*_{k,h}(g+1)| \ge k^2 \E(|\mathcal B^*_{k,h}(g+1)|)) \le \frac{1}{k^2}.$$

Since $|\mathcal B^*_{k,h}(g+1)| = 0$ for $h^{k+1}<m$ we have
$$ \PP \left (|\mathcal B^*_{k,h}(g+1)| \ge k^2 \E(|\mathcal B^*_{k,h}(g+1)|) \text{ for some } k \right ) \le
\sum_{k \ge \log_h m - 1} \frac{1}{k^2} = O \left( \frac{1}{\log m}
\right),$$ so with probability $1 - O \left( \frac{1}{\log m}
\right)$ we have that $|\mathcal B^*_{k,h}(g+1)| \ll k^2 h^{k((h - 1
- h \alpha)(g+1) + 1)}$ for every $k$.

On the other hand, by Theorem \ref{densi} we know that $A(x) \gg x^{1 - \alpha}$ with probability 1.

So, given $x$, we will have $h^l \le x < h^{l+1}$ for some $l$ and with probability $1 - O \left( \frac{1}{\log m} \right)$ the number of bad elements less than or equal to $x$ will be
$$\mathcal B(x) \le \sum_{k=0}^{l} |\mathcal B^*_{k,h}(g+1)| \ll l^2 h^{l ((h - 1 - h \alpha)(g+1) + 1)}$$
while the number of elements in $A$ less than or equal to $x$ will be
$$A(x) \gg h^{l(1 - \alpha)}.$$

Since, in order for $A$ to be in $\tilde{B}^*_h[g]$, we want $\mathcal B(x) = o(A(x))$ we need
$$(h - 1 - h \alpha)(g+1) + 1 < 1 - \alpha$$
and so, with $\alpha = \frac{2h - 4}{2h - 3} + \delta$ we have
$$g > \frac{\frac{h-1}{2h-3} - (h-1)\delta}{\frac{h - 3}{2h - 3} + h \delta}.$$
\end{proof}

In particular, for $h=2$, since a $\tilde{B}^*_2[g]$ sequence is also a $\tilde{B}_2[g]$ sequence, we deduce that given
$0 < \e < \frac{1}{3}$, a random sequence $A$ in $S \left( \frac{2}{3} + \e, m \right)$ is $\tilde{B}_2[1]$
with probability $1 - O \left( \frac{1}{\log m} \right)$.

\smallskip

Also, for $h=3$, we deduce that given $0 < \delta < \frac{1}{3}$, a random sequence $A$ in $S \left( \frac{2}{3} + \delta, m \right)$ is
 $\tilde{B}^*_3[g]$ for every $g > \frac{2}{9 \delta} - \frac{2}{3}$ with probability $1 - O \left( \frac{1}{\log m} \right)$.

\smallskip

Lemma \ref{captilde} gives the proof of Theorem \ref{boundge}.

\end{document}